\documentclass[a4paper,reqno]{amsart}

\usepackage{enumerate}
\usepackage[usenames]{color}

\usepackage[colorlinks=true]{hyperref}
\hypersetup{urlcolor=blue, linkcolor=blue, citecolor=red}

\numberwithin{equation}{section}

\newtheorem{thm}{Theorem}[section]

\newtheorem{prop}[thm]{Proposition}

\theoremstyle{definition}
\newtheorem{rem}[thm]{Remark}

\newcommand\R{{\mathbb R}}
\newcommand\C{{\mathbb C}}

\newcommand\im{{\mathrm{Im}}\,}
\newcommand\Tma{T_{\mathrm{max}}}

\newcommand\Cz{{C_0(\R^N )}}

\newcommand\DI{u_0 }

\newcommand\Tsem{{\boldsymbol{\mathcal T}}}
\newcommand\Est{{ e^\star }}
\newcommand\Cgn{K}

\newcommand\MScN[1]{\href{http://www.ams.org/mathscinet-getitem?mr=#1}{\nolinkurl{(#1)}}}
\newcommand\DOI[1]{\href{http://dx.doi.org/#1}{(doi: \nolinkurl{#1})}}
\newcommand\LINK[1]{\href{#1}{(link: \nolinkurl{#1})}}

\title{Finite-time blowup for a complex Ginzburg-Landau equation with linear driving}

\author[Thierry Cazenave]{Thierry Cazenave$^1$}
\author[Jo\~ao Paulo Dias]{Jo\~ao Paulo Dias$^{\dag,2}$}
\author[M\'ario Figueira]{M\'ario Figueira$^{\dag,2}$}

\address{$^1$Universit\'e Pierre et Marie Curie \& CNRS, Laboratoire Jacques-Louis Lions,
B.C. 187, 4 place Jussieu, 75252 Paris Cedex 05, France}

\address{$^{2}$Centro de Matem\'atica e Aplica\c c\~oes Fundamentais, Universidade de Lisboa, Avenida Prof. Gama Pinto 2, 1649--003 Lisboa, Portugal}

\email[Thierry Cazenave]{\href{mailto:thierry.cazenave@upmc.fr}{thierry.cazenave@upmc.fr}}
\email[Jo\~ao Paulo Dias]{\href{mailto:dias@ptmat.fc.ul.pt}{dias@ptmat.fc.ul.pt}}
\email[Zheng Han]{\href{mailto:figueira@ptmat.fc.ul.pt}{figueira@ptmat.fc.ul.pt}}

\subjclass[2010] {35Q56, 35B44}

 \keywords{Complex Ginzburg-Landau equation, damping, finite-time blowup, energy, variance}

\thanks{$^\dag$ Research partially supported by the Portuguese Foundation for Science and Technology (FCT) through the grant PTDC/MAT/110613/2009 and by PEstOE/MAT/UI 0209/2011.}

\begin{document}

\begin{abstract}
In this paper, we consider the complex Ginzburg--Landau equation $u_t =  e^{ i\theta } [\Delta u +  |u|^\alpha u] + \gamma u$ on $\R^N $, where $\alpha >0$, $\gamma \in \R$ and  $-\pi /2<\theta <\pi /2$.
By convexity arguments we prove that, under certain conditions on $\alpha ,\theta ,\gamma $, a class of solutions with negative initial energy  blows up in finite time. 
\end{abstract}

\maketitle

\section{Introduction} \label{Intro} 
In this article, we consider the complex Ginzburg-Landau equation
\begin{equation} \label{GL} 
\begin{cases} 
 u_t =  e^{ i\theta } [\Delta u +  |u|^\alpha u] + \gamma u, \\
u(0)= \DI,
\end{cases} 
\end{equation} 
in $\R^N $, where $ -\frac {\pi } {2} \le \theta \le \frac {\pi } {2}$,  $\alpha >0$ and $\gamma \in \R$\footnote{In principle, one could let $\gamma $ be complex, but the imaginary part can be eliminated by the change of variable $v(t,x)= e^{ - i t \Im \gamma } u(t,x)$.}, and we look for conditions on the initial value $\DI$ and the parameters $\theta ,\alpha $ and $\gamma $ that ensure finite-time blowup of the solution.
Equation~\eqref{GL}  is a particular case of the more general complex Ginzburg--Landau equation
\begin{equation} \label{GGL} 
 u_t= e^{i\theta }\Delta u + e^{i\phi } |u|^\alpha u+ \gamma u, 
\end{equation} 
 which is used to model such phenomena as  superconductivity, chemical turbulence, and various types of fluid flows;
see~\cite{DoeringGHN} and the references cited therein. 
Note that the solutions of equation~\eqref{GL} satisfy certain energy identities (see Section~\ref{Cauchy}), which are not shared by the solutions of~\eqref{GGL}. 

Equation~\eqref{GL} with  $\theta =0$ is nonlinear heat equation 
\begin{equation} \label{NLH} 
\begin{cases} 
u_t-\Delta u= |u|^\alpha u +\gamma u, \\ u(0)= \DI,
\end{cases} 
\end{equation} 
while for $\theta = \pm \pi /2$, \eqref{GL} reduces to the nonlinear  Schr\"o\-din\-ger equation 
\begin{equation} \label{NLS} 
\begin{cases} 
u_t = \pm i( \Delta u+  |u|^\alpha u) +\gamma u , \\ u(0)= \DI.
\end{cases} 
\end{equation} 
In particular, \eqref{GL} is ``intermediate" between the nonlinear heat and Schr\"o\-din\-ger equations.

We recall (see Section~\ref{Cauchy}) that the Cauchy problem~\eqref{GL} is locally well-posed in $H^1 (\R^N ) \cap \Cz$, where $\Cz$ is the space of  continuous functions $\R^N \to \C$ which vanish at infinity, equipped with the sup norm. 
In particular, given any  $\DI \in H^1 (\R^N ) \cap \Cz$, there exists a unique solution $u$ of~\eqref{GL} defined on a maximal interval $[0, \Tma)$, i.e., $u\in C([0, \Tma ), H^1 (\R^N ) \cap \Cz)$. If the maximal existence time $\Tma $ is finite, then the solution blows up at $\Tma$ in $  \Cz$.

The effect of the driving term $\gamma u$ can be easily seen on the ODE associated with~\eqref{GL}, i.e.,
\begin{equation} \label{fODE} 
v' = e^{ i\theta }   |v|^\alpha v + \gamma v.
\end{equation}  
The solution of~\eqref{fODE} with the initial condition $v(0)= v_0\in \C$, is given by 
\begin{equation}  \label{fODE2} 
v(t)= e^{\gamma t}  \Bigl[ 1- \frac {e^{\alpha \gamma t}-1} {\gamma }  |v_0|^\alpha \cos \theta   \Bigr]^{- \frac {1} {\alpha } (1+i \tan \theta )} v_0, 
\end{equation} 
as long as this formula makes sense. (The term $ \frac {e^{\alpha \gamma t}-1} {\gamma }$ must be replaced by $\alpha t$ if $\gamma =0$.)
If $\gamma =0$, then we see that for every $v_0\not = 0$, the solution blows up in finite time. The same conclusion holds if $\gamma >0$. On the other hand, when $\gamma <0$, whether or not $v$ blows up depends on the size of $ |v_0|$. More precisely, if $ |v_0| > \frac {-\gamma } {\cos \theta }$, then $v$ blows up in finite time, whereas if $ |v_0| \le \frac {-\gamma } {\cos \theta }$, then $v$ is global.

When   $\gamma =0$, finite-time blowup for equation~\eqref{GL}  is known to occur under a negative energy condition. More precisely, let the energy $E$ be defined by
\begin{equation} \label{fEn1} 
E(w)= \frac {1} {2}\int  _{ \R^N  } |\nabla w|^2 - \frac {1} {\alpha +2} \int  _{ \R^N  } |w|^{\alpha +2},
\end{equation} 
for $w\in H^1 (\R^N ) \cap \Cz$. 
If $\theta =0$, then it follows from Levine~\cite{Levine} that the solution of the nonlinear heat equation~\eqref{NLH} blows up in finite time if $\DI\in H^1 (\R^N ) \cap \Cz$ satisfies $E(\DI) <0$. 
For $-\frac {\pi } {2}< \theta < \frac {\pi } {2}$,  negative energy initial values also yield finite-time blowup for the equation~\eqref{GL}, see~\cite{CazenaveDW}. 
If $\theta =\pm \frac {\pi } {2}$, then the solution of the nonlinear Schr\"o\-din\-ger equation~\eqref{NLS} blows up in finite time provided $\frac {4} {N} \le  \alpha <\frac {4} {N-2}$ and the initial value $\DI\in H^1 (\R^N ) $ satisfies $E(\DI) <0$ and $ |\cdot |\DI\in L^2 (\R^N ) $. (See Zakharov~\cite{Zakharov} and Glassey~\cite{Glassey}.) 

If   $\gamma >0$, obvious modifications of the arguments used when $\gamma =0$ provide similar results. In particular, if the initial value $\DI\in H^1 (\R^N ) \cap \Cz$ satisfies $E (\DI) <0$, then the corresponding solution of~\eqref{NLH} blows up in finite time. Moreover, if $\frac {4} {N} \le  \alpha <\frac {4} {N-2}$ and the initial value $\DI\in H^1 (\R^N ) $ satisfies $E(\DI) <0$ and $ |\cdot |\DI\in L^2 (\R^N ) $, then the solution of~\eqref{NLS} blows up in finite time. The situation is similar for general $-\frac {\pi } {2}< \theta <\frac {\pi } {2}$, and a simple modification of the argument of~\cite{CazenaveDW} shows finite-time blowup for initial values with negative energy. More precisely, we have the following result.

\begin{thm} \label{ePos1} 
Assume
\begin{equation} \label{fGGLuA} 
-\frac {\pi } {2} < \theta < \frac {\pi } {2},
\end{equation} 
 $\alpha >0$ and $\gamma >0$. 
Let $\DI\in H^1 (\R^N ) \cap \Cz$ and let $u\in C([0, \Tma), H^1 (\R^N ) \cap  \Cz)$ be the corresponding maximal solution of~\eqref{GL}.  If $E( \DI) <0$, 
where $E$ is defined by~\eqref{fEn1}, then $u$ blows up in finite time, i.e., $\Tma <\infty $. 
\end{thm} 

When $\gamma <0$, the situation is more delicate.
For the nonlinear heat equation~\eqref{NLH}, Levine's calculations~\cite{Levine} can be adapted in order to show that if the initial value $\DI\in H^1 (\R^N ) \cap \Cz$ satisfies $E_\gamma (\DI) <0$, where
\begin{equation*} 
E_\gamma (w)=  \frac {1} {2}\int  _{ \R^N  } |\nabla w|^2 - \frac {1} {\alpha +2} \int  _{ \R^N  } |w|^{\alpha +2} - \frac {\gamma } {2} \int  _{ \R^N  } |w|^{2},
\end{equation*} 
then the solution of~\eqref{NLH} blows up in finite time. (See~\cite[Theorem~17.6]{QuittnerS}.) 
For the nonlinear Schr\"o\-din\-ger equation~\eqref{NLS}, Glassey's proof~\cite{Glassey} is not immediately applicable. Sufficient conditions for finite-time blowup were obtained by M.~Tsutsumi~\cite{MTsu} (see also~\cite{OhtaT}) by a delicate modification of the variance argument of~\cite{Glassey}. It follows in particular from the calculations in~\cite{MTsu, OhtaT} that if $\frac {4} {N} < \alpha <\frac {4} {N-2}$ and the initial value $\DI\in H^1 (\R^N ) $ satisfies
\begin{equation*} 
E(\DI) + \frac {\alpha \gamma } {N\alpha -4} \im \int  _{ \R^N  }\DI (x\cdot \nabla  \overline{\DI} ) + \frac {\alpha ^2\gamma ^2} {(N\alpha -4)^2} \int  _{ \R^N  } |x|^2  |\DI |^2 <0,
\end{equation*} 
then the solution of~\eqref{NLS} blows up in finite time. Note that the above condition becomes stronger and stronger as $\alpha \downarrow \frac {4} {N}$, and that no energy-type sufficient condition is known for blowup if $\alpha =\frac {4} {N}$. (The case $\alpha =\frac {4} {N}$ is studied in~\cite{Mohamad} by a very different method.)

For the equation~\eqref{GL} with $\gamma <0$, we have the following result. 

\begin{thm} \label{eMain1} 
Assume
\begin{equation} \label{eHyp1}
-\frac {\pi } {4} < \theta  < \frac {\pi } {4}, 
\end{equation} 
$\alpha >0$ and $\gamma <0$. 
Let $\DI\in H^1 (\R^N ) \cap \Cz$ and let $u\in C([0, \Tma), H^1 (\R^N ) \cap  \Cz)$ be the corresponding maximal solution of~\eqref{GL}.  Suppose further that
\begin{equation}  \label{eHyp2} 
(\alpha +2) \cos (2\theta ) + 2 (1-\cos \theta )\ge 2 \cos \theta. 
\end{equation} 
If
\begin{equation} \label{eHyp3} 
E( \DI) + \xi  \int  _{ \R^N  }  | \DI |^2 <0,
\end{equation} 
where $E$ is defined by~\eqref{fEn1} and 
\begin{equation} \label{eHyp4} 
\xi  = - \frac {\gamma } {\cos \theta } \max  \Bigl\{ \frac {1} {\alpha }, \frac {(\alpha +2) \cos (2\theta ) + 2 (1-\cos \theta )} {2} \Bigr\},
\end{equation} 
then $u$ blows up in finite time, i.e., $\Tma <\infty $. 
\end{thm} 

Theorem~\ref{eMain1} calls for several comments. First, assume $\alpha $ and $\theta $ satisfy~\eqref{eHyp1}-\eqref{eHyp2} and let $\psi  \in H^1 (\R^N ) \cap \Cz$, $\psi \not = 0$. It follows that $\DI = \kappa  \psi $, with $k\in \C$, satisfies~\eqref{eHyp3} provided $ |\kappa | $ is sufficiently large. 

Next, assumption~\eqref{eHyp1} means that equation~\eqref{GL} is not (formally) close to the nonlinear Schr\"o\-din\-ger equation~\eqref{NLS}. Assuming~\eqref{eHyp1},  we see that~\eqref{eHyp2} is satisfied for a fixed $\theta $ if $\alpha $ is sufficiently large. Alternatively, \eqref{eHyp2} is satisfied for a fixed $\alpha >0$ if $ |\theta | $ is sufficiently small.

The assumptions of Theorem~\ref{eMain1} are probably not optimal, since letting $\gamma \downarrow 0$ yields the ``natural" condition $E ( \DI) <0$, but also the structural conditions~\eqref{eHyp1}-\eqref{eHyp2}. In particular, Theorem~\ref{eMain1} does not include the result of~\cite{CazenaveDW}. 
On the other hand, note that if $\gamma <0$, then there does not exist any map $F:H^1\cap C_0\to \R$ such that if $F( \DI ) <0$, then the solution of~\eqref{GL}  blows up in finite time for all $\theta \in (-\frac {\pi } {2}, \frac {\pi } {2})$. (At least if $\alpha <4/N$.) Indeed, given any $\alpha <4/N$ and any  $ \DI \in H^1 (\R^N ) \cap \Cz$, it follows from Remark~\ref{eCauchy2} that the solution of~\eqref{GL}  is global provided $\theta $ is sufficiently close to $\pm \frac {\pi } {2}$. 
This is in sharp contrast with the case $\gamma =0$, where negative energy yields finite-time blowup for every $ -\frac {\pi } {2} < \theta < \frac {\pi } {2}$. 

Note that finite-time blowup of certain solutions (in $L^\infty $, not necessarily of finite energy) of~\eqref{GGL} is proved in~\cite{MasmoudiZ} under the structural assumptions $-\frac {\pi } {2}< \theta , \phi < \frac {\pi } {2}$ and $\tan ^2 \phi + (\alpha +2) \tan \theta \tan \phi <\alpha +1$. 
For the equation~\eqref{GL}, the last assumption reduces to $\tan ^2 \theta <\frac {\alpha +1} {\alpha +3}$, i.e.,
\begin{equation} \label{fMZ} 
(\alpha +2) \cos ^2 \theta > \frac {\alpha +3} {2}.
\end{equation}  
(Note that $\tan ^2 \theta <1$ so that in particular $\theta $ satisfies~\eqref{eHyp1}.)
On the other hand, condition~\eqref{eHyp2} is equivalent to
\begin{equation} \label{fMZ2} 
(\alpha +2) \cos ^2 \theta \ge  \frac {\alpha } {2}+ 2 \cos \theta .
\end{equation} 
Conditions~\eqref{fMZ} and~\eqref{fMZ2} are  not comparable. In particular, condition~\eqref{fMZ} is stronger if $\theta $ is close to $\pm \frac {\pi } {4}$, whereas condition~\eqref{fMZ2} is stronger if $\theta $ is close to $0$. 

The rest of this paper is organized as follows. In Section~\ref{Cauchy}, we recall some simple properties of the Cauchy problem~\eqref{GL}. In Section~\ref{sPos} we prove Theorem~\ref{ePos1}, and in Section~\ref{sBUp} we prove Theorem~\ref{eMain1}.

\section{The Cauchy problem} \label{Cauchy} 
Assume~\eqref{fGGLuA}.
It is well known that the operator $e^{i\theta } \Delta $ with domain $H^2 (\R^N ) $
 generates an analytic semigroup
of contractions $(\Tsem _\theta (t) )_{ t\ge 0 }$ on $L^2 (\R^N ) $. 
Moreover,
\begin{equation} \label{fDispu} 
 \|\Tsem _\theta (t) \psi \| _{ L^r } \le  (\cos \theta )^{-\frac {N} {2} (1- \frac {1} {p}+ \frac {1} {r})}
 t^{-\frac {N} {2} (  \frac {1} {p}- \frac {1} {r})}  \|\psi \| _{ L^p },
\end{equation} 
for $1\le p\le r\le \infty $ and $\theta $ satisfying~\eqref{fGGLuA} and
$(\Tsem _\theta (t) )_{ t\ge 0 }$ is a bounded $C_0$ semigroup on $L^p (\R^N ) $ for $1\le p<\infty $ and on $\Cz$.
Moreover, equation~\eqref{GL} can be written in the equivalent integral form
\begin{equation} \label{IGL} 
u(t)= \Tsem_\theta (t) \DI +  \int _0^t \Tsem_\theta (t-s)  [e^{i\theta } |u(s)|^\alpha u(s)
+ \gamma u(s)]\,ds .
\end{equation} 
It is immediate by applying a contraction mapping argument to~\eqref{IGL}  that the Cauchy problem~\eqref{GL} is locally well  posed in $\Cz$. Moreover, it is easy to see using the estimates~\eqref{fDispu} that $\Cz \cap H^1 (\R^N ) $ is preserved under the action of~\eqref{GL}. 
More precisely, we have the following result.

\begin{prop} \label{eGGLzu} 
Suppose~\eqref{fGGLuA}, $\alpha >0$ and $\gamma \in \R$.  
Given any $\DI \in \Cz \cap H^1 (\R^N ) $, there exist $T>0$ and a unique function $u \in  C([0,T], \Cz
\cap H^1 (\R^N ) ) 
\cap C((0,T), H^2 (\R^N ) )\cap C^1((0, T), L^2 (\R^N ) ) $ which satisfies~\eqref{GL} for all $t\in (0,T)$
and such that $u(0)= \DI$. 
Moreover, $u$ can be extended to a maximal interval $[0, \Tma)$, and if $\Tma <\infty $, then $  \|u(t)\| _{ L^\infty  }\to \infty $ as $t \uparrow \Tma$.
\end{prop} 

\begin{rem}  \label{eGGLzd} 
{\rm Let $\DI \in \Cz \cap H^1 (\R^N ) $ and  $u$  the corresponding solution of~\eqref{GL} defined on the maximal interval $[0, \Tma)$, and given by Proposition~$\ref{eGGLzu}$.
If, in addition, $\alpha <4/N$, then~\eqref{GL} is locally well posed in $L^2 (\R^N ) $ (see~\cite{Weissler}). It is not difficult to show using the estimates~\eqref{fDispu} that the maximal existence times in $H^1 (\R^N ) \cap \Cz$ and $L^2 (\R^N ) $ are the same;
 and so if $\Tma <\infty $, then $  \|u(t)\| _{ L^2  }\to \infty $ as $t \uparrow \Tma$.}
\end{rem} 

We collect below the energy identities that we use in the next sections.

\begin{prop} \label{eGGLu} 
Suppose~\eqref{fGGLuA}, $\alpha >0$ and $\gamma \in \R$. If $\DI \in \Cz \cap H^1 (\R^N ) $ and and $u$ is the corresponding  solution of~\eqref{GL}  defined on the maximal interval $[0, \Tma)$, then the following properties hold. 

\begin{enumerate} [{\rm (i)}]

\item \label{eGGLu:d}
Set 
\begin{equation}  \label{fGGLc} 
I(w)=  \int _{\R^N }  |\nabla w|^2 - \int _{\R^N }  |w|^{\alpha +2},
\end{equation} 
for $w\in \Cz \cap H^1  (\R^N ) $. It follows that
\begin{equation} \label{fGGLsbumu}
 \int _{\R^N }  \overline{u}  u_t =  \gamma \int  _{ \R^N  }  |u|^2 - e^{i\theta } I(u).
\end{equation} 
In particular,
\begin{gather} 
 | I(u)| = \Bigl| \int _{\R^N }   \overline{u} u_t -  \gamma \int  _{ \R^N  }  |u|^2  \Bigr| , \label{fGGLsbu} \\
|I(u(t))|^2   =  \Bigl| \int  _{ \R^N  } \overline{u}u_t  \Bigr|^2 +\gamma ^2  \Bigl( \int  _{ \R^N  }  |u|^2 \Bigr)^2 - \gamma  \Bigl(  \int  _{ \R^N  }  |u|^2 \Bigr) \frac {d} {dt} \int  _{ \R^N  }  |u|^2, \label{fGGLsbu1}
\end{gather} 
and 
\begin{equation} \label{fGGLs}
\frac {d} {dt} \int _{\R^N }  |u|^2= 2\gamma \int  _{ \R^N  }  |u|^2  - 2\cos \theta  \ I(u),
\end{equation} 
for all $0< t< \Tma$. 

\item \label{eGGLu:u}
If $E$ is defined by~\eqref{fEn1}, then
\begin{equation} \label{fGGLqbu}
\frac {d} {dt} E(u(t)) = 
 - \cos \theta \int  _{ \R^N  } |u_t|^2 + \gamma ^2 \cos \theta \int  _{ \R^N  } |u|^2 - \gamma \cos (2 \theta ) I(u),
\end{equation} 
and
\begin{equation} \label{fGGLq}
\frac {d} {dt}  \Bigl[ E(u(t)) - \frac {\gamma } {2} \cos \theta \int  _{ \R^N  } |u|^2 \Bigr] = 
 - \cos \theta \int  _{ \R^N  } |u_t|^2 + \gamma \sin ^2 \theta \ I(u),
\end{equation} 
for all $0< t< \Tma$. 

\end{enumerate} 

\end{prop} 

\begin{proof} 
Multiplying equation~\eqref{GL}  by  $ \overline{u} $ and integrating by parts on $\R^N  $ yields~\eqref{fGGLsbumu}, and identities~\eqref{fGGLsbu}, \eqref{fGGLsbu1} and~\eqref{fGGLs} are immediate consequences. 
Multiplying equation~\eqref{GL}  by  $e^{-i\theta } \overline{u}_t $, integrating by parts on $\R^N  $ and taking the real part, we obtain
\begin{equation} \label{fGLP1} 
\begin{split} 
\frac {d} {dt} E(u(t)) & = - \cos \theta \int  _{ \R^N  } |u_t|^2 + \gamma \Re  \Bigl( e^{ i\theta } \int  _{ \R^N  }   \overline{u} u_t  \Bigr) \\ & 
= - \cos \theta \int  _{ \R^N  } |u_t|^2 + \gamma   \Bigl( \cos \theta \ \Re \int  _{ \R^N  }   \overline{u} u_t  - \sin \theta \ \Im \int  _{ \R^N  }   \overline{u} u_t \Bigr).
\end{split} 
\end{equation} 
Furthermore, it follows from~\eqref{fGGLsbumu} that 
\begin{equation} \label{fGLP3}
 \Re  \Bigl( e^{ i\theta } \int  _{ \R^N  }   \overline{u} u_t  \Bigr) =    \gamma\cos \theta \int  _{ \R^N  }  |u|^2 - \cos (2\theta ) I(u) .
\end{equation} 
Identity~\eqref{fGGLqbu} follows from the first identity in~\eqref{fGLP1}, and~\eqref{fGLP3}.
On the other hand, taking the imaginary part of~\eqref{fGGLsbumu}, we obtain
\begin{equation} \label{fGLP2}
 \Im \int  _{ \R^N  }   \overline{u} u_t = - \sin \theta \ I(u). 
\end{equation} 
Identity~\eqref{fGGLq} now follows from the second identity in~\eqref{fGLP1}, and~\eqref{fGLP2}.
\end{proof} 

\begin{rem} 
Note that
\begin{equation} \label{fGLP8} 
I(w)= (\alpha +2) E(w) - \frac {\alpha } {2}\int  _{ \R^N  } |\nabla u|^2 \le (\alpha +2) E(w).
\end{equation} 
\end{rem} 

We conclude this section with a global existence property for sufficiently small initial values in the case $\gamma <0$. 

\begin{prop} \label{eCauchy1} 
There exists a constant $\Cgn>0$ with the following property. 
Given $\gamma <0$, $0< \alpha <\frac {4} {N}$, and $\Cz \cap H^1 (\R^N ) $, let $u\in C([0, \Tma),  H^1 (\R^N ) \cap \Cz )$ be the corresponding, maximal solution of~\eqref{GL}. 
If 
\begin{equation}  \label{fGlob2:7}
 \| \DI \| _{ L^2 } \le  \Biggl[  \frac {4 |\gamma |} {(4-N\alpha )  (  \frac {N\alpha } {4} 
 \Cgn ) ^{  \frac {N\alpha } {4-N\alpha } } \Cgn \cos \theta   } \Biggr]^{\frac  {4- N\alpha }{4\alpha }},
\end{equation} 
then $u$ is global, i.e., $\Tma= \infty $.
\end{prop} 

\begin{proof} 
For $0\le t<\Tma$, set
\begin{equation} \label{fGlob2:1} 
f(t)= \int  _{ \R^N  } |u(t,x)|^2 dx.
\end{equation} 
Recall that (by Sobolev's or Gagliardo-Nirenberg's inequality) that there exists a constant $\Cgn$ such that 
\begin{equation} \label{fGlob2:3}
 \| w \| _{ L^{\alpha +2} }^{\alpha +2} \le \Cgn \|\nabla  w \| _{ L^2 }^{\frac {N\alpha } {2}}  \| w \| _{ L^2 } ^{ \frac {4- (N-2) \alpha } {2}},
\end{equation} 
for all $0\le \alpha \le \frac {4} {N}$ and all $w\in H^1 (\R^N ) $.
Applying the elementary inequality
$xy \le \frac {\varepsilon ^p x^p } {p}+ \frac {y^{p'}} {p' \varepsilon ^{p'}} $ 
with $p= \frac {4} {N\alpha }$, we deduce from~\eqref{fGlob2:3} that 
\begin{equation}  \label{fGlob2:5}
 \| w \| _{ L^{\alpha +2} }^{\alpha +2} \le \Cgn  
  \Bigl( \frac {N\alpha } {4} \varepsilon ^{\frac {4} {N\alpha }}  \|\nabla w \| _{ L^2 }^2 + \frac {4-N\alpha } {4} \varepsilon ^{ - \frac {4} {4-N\alpha } }  \| w \| _{ L^{2} }^{\frac {2 [ 4-(N-2)\alpha ]} {4-N\alpha }} \Bigr).
\end{equation}
It now follows from~\eqref{fGlob2:1},  \eqref{fGGLs} and~\eqref{fGlob2:5} that
\begin{equation*} 
\frac {df} {dt}  \le 2\gamma f + \cos \theta \ \Cgn  \varepsilon ^{ - \frac {4 } {4-N\alpha} } \frac {4-N\alpha } {2}  f^{\frac { 4-(N-2)\alpha } {4-N\alpha }} 
- 2\cos \theta  \Bigl[ 1 -   \varepsilon ^{\frac {4 } {N\alpha}} \frac {N\alpha } {4} \Cgn  \Bigr]  \|\nabla u\| _{ L^2 }^2.
\end{equation*} 
Letting
$\varepsilon =  (  \frac {N\alpha } {4}\Cgn )^{- \frac {N\alpha} {4 }}$,
 we deduce that
\begin{equation} \label{fGlob2:5b}
\frac {df} {dt} \le 2\gamma f + \cos \theta \ \Cgn  \Bigl(  \frac {N\alpha } {4} \Cgn \Bigr) ^{  \frac {N\alpha } {4-N\alpha } } \frac {4-N\alpha } {2}   f^{\frac { 4-(N-2)\alpha } {4-N\alpha }}.
\end{equation} 
This is an inequality of the form
$f'+ af \le b f^{1+\nu }$. If 
$ f (0)^{\nu } \le \frac {a} {b}$,
then this implies $f (t) \le  e^{-a t} ( f (0)^{-\nu } -  \frac {b} { a} )^{-\frac {1} {\nu }} $.
Therefore, it follows from~\eqref{fGlob2:5b}  and~\eqref{fGlob2:7} that $\sup _{ 0\le t<\Tma } \| u(t)\| _{ L^2 }< \infty $. Applying Remark~\ref{eGGLzd}, we conclude that $\Tma = \infty $.
\end{proof} 

\begin{rem} \label{eCauchy2} 
For a fixed $\alpha $, condition~\eqref{fGlob2:7} becomes better and better as $\theta \to \pm \frac {\pi } {2}$. (This is not too surprising. Indeed, for the limiting nonlinear Schr\"o\-din\-ger equation, global existence holds for every initial value.) 
More precisely, the right-hand side of~\eqref{fGlob2:7} goes to $\infty $ as $\theta \to \pm \frac {\pi } {2}$. 
In particular, if we fix $\alpha <\frac {4} {N}$, $\gamma <0$ and an initial value $ \DI $, then the solution of~\eqref{GL}  is global if $\theta $ is sufficiently close to $\pm \frac {\pi } {2}$.
\end{rem} 

\section{Proof of Theorem~$\ref{ePos1}$} \label{sPos} 
We follow the argument of~\cite{CazenaveDW} after an appropriate change of variables.
Set 
\begin{equation}  \label{fPos1} 
v(t)= e^{- \gamma t} u(t), 
\end{equation} 
for $0\le t<\Tma$. (Note that $v_t= e^{i\theta } [\Delta v+ e^{\alpha \gamma t} |v|^\alpha v] $.)
Set
\begin{gather}  
 \widetilde{f}(t) =\int  _{ \R^N  } |v|^2 = e^{-2\gamma t}\int  _{ \R^N  } |u|^2,  \label{fPos3} \\
 \widetilde{\jmath} (t)=  \int _{\R^N }  |\nabla v|^2 - e^{\alpha \gamma t} \int _{\R^N }  |v|^{\alpha +2} 
 = e^{-2\gamma t} I(u),
 \label{fPos4} 
\end{gather} 
and 
\begin{equation}  \label{fPos5} 
 \widetilde{e} (t )= \frac {1} {2}\int _{\R^N }  |\nabla v|^2 -\frac {e^{\alpha \gamma t}} {\alpha +2}\int _{\R^N }  |v|^{\alpha +2} = e^{-2\gamma t} E(u).
\end{equation} 
It follows from~\eqref{fGGLsbumu}, \eqref{fPos1} and~\eqref{fPos4}  that
\begin{equation} \label{fPos6b} 
\int  _{ \R^N  }  \overline{v}v_t= -e^{i\theta }   \widetilde{\jmath} (t),
\end{equation} 
so that by~\eqref{fPos3}
\begin{equation} \label{fPos6}
\frac {d \widetilde{f} } {dt}  = 2 \Re \int _{\R^N }  \overline{v}  v_t =   - 2 \cos \theta \   \widetilde{\jmath } (t).
\end{equation} 
Moreover, it follows from~\eqref{fPos5}, \eqref{fGGLqbu},  \eqref{fPos1},  \eqref{fPos6} and~\eqref{fPos4} that
\begin{equation} \label{fPos7}
\frac {d \widetilde{e} } {dt}  =   - \cos \theta   \int  _{ \R^N  } |v_t|^2  - e^{\alpha \gamma t} \frac {\alpha \gamma } {\alpha +2} \int  _{ \R^N  } |v|^{\alpha +2} \le   - \cos \theta   \int  _{ \R^N  } |v_t|^2.
\end{equation} 
Since $ \widetilde{e} (0)= E( \DI) <0$, we deduce from~\eqref{fPos7} that 
\begin{equation} \label{fPos7b}
 \widetilde{e} (t)<0,
\end{equation} 
for all $0\le t<\Tma$. 
It follows from~\eqref{fPos7}, Cauchy-Schwarz's inequality, \eqref{fPos6b} and~\eqref{fPos6} that
\begin{equation}  \label{fPos8}
\begin{split} 
-  \widetilde{f}  \frac {d \widetilde{e} } {dt} & \ge \cos \theta   \int  |v|^2 \int  |v_t|^2 \ge \cos \theta   \Bigl| \int  \overline{v}v_t  \Bigr|^2 \\ & = \cos \theta \    \widetilde{\jmath }  ^2  = \frac {1} {2}   (-  \widetilde{\jmath }   ) \frac {d \widetilde{f} } {dt} .
\end{split} 
\end{equation} 
On the other hand, note that
\begin{equation}  \label{fPos9}
 \widetilde{\jmath }  = (\alpha +2)  \widetilde{e}  -\frac {\alpha } {2} \int_{ \R^N  } |\nabla v|^2 \le  (\alpha +2)  \widetilde{e} .
\end{equation} 
It follows from~\eqref{fPos7b} and~\eqref{fPos9}  that $ \widetilde{\jmath }  <0$, so that by~\eqref{fPos7}, $\frac {d \widetilde{f} } {dt} >0$; and so, we deduce from~\eqref{fPos8}-\eqref{fPos9} that
\begin{equation} \label{fPos10}
-  \widetilde{f}  \frac {d \widetilde{e} } {dt}\ge -\frac {\alpha +2} {2}  \widetilde{e}   \frac {d  \widetilde{f} } {dt}.
\end{equation} 
Therefore,
\begin{equation}  \label{fPos11}
\frac {d} {dt} [-  \widetilde{e}  \widetilde{f}^{-\frac {\alpha +2} {2}}]\ge 0,
\end{equation} 
so that
\begin{equation}  \label{fPos12}
- \widetilde{e} (t) \ge [- \widetilde{e} (0)] \widetilde{f}(0)^{-\frac {\alpha +2} {2}} \widetilde{f}(t)^{\frac {\alpha +2} {2}} = (- E  (\DI) ) \|\DI\| _{ L^2 } ^{-(\alpha +2)}  \widetilde{f} (t)^{\frac {\alpha +2} {2}}.
\end{equation} 
It now follows from~\eqref{fPos6}, \eqref{fPos9} and~\eqref{fPos12} that
\begin{equation}  \label{fPos13}
\begin{split} 
\frac {d \widetilde{f}} {dt} & = - 2 \cos \theta \   \widetilde{\jmath }   \ge -2(\alpha +2)   \cos \theta \  \widetilde{e} \\ &  \ge 2(\alpha +2)  \cos \theta \   (- E (\DI) ) \|\DI\| _{ L^2 } ^{-( \alpha +2) } \widetilde{f}^{\frac {\alpha +2} {2}},
\end{split} 
\end{equation} 
which implies
\begin{equation}  \label{fPos14}
\frac {d} {dt}  \{\alpha  (\alpha +2)  \cos \theta \ (- E (\DI) ) \|\DI\| _{ L^2 } ^{-( \alpha +2) }  t +   \widetilde{f}^{- \frac {\alpha } {2}} \} \le 0.
\end{equation} 
Since~\eqref{fPos14} holds for all $0\le t<\Tma$, we deduce that
\begin{equation} \label{fPos15}
\Tma \le \frac { \| \DI \| _{ L^2 }^2} {\alpha (\alpha +2)  \cos \theta \  (- E(\DI )) } <\infty .
\end{equation} 
This completes the proof.

\section{Proof of Theorem~$\ref{eMain1}$}\label{sBUp} 

Consider $\DI$ as in the statement and $u$ the corresponding solution of~\eqref{GL} defined on the maximal interval $[0, \Tma)$.  
We first show that a certain energy of $u$ remains negative as long as $u$ exists. Then, we use this property in order to derive a differential inequality which shows that $u$ cannot be global. 

It is convenient to set
\begin{equation} \label{fBUp1z:1}
\rho = - \gamma >0,
\end{equation} 
and 
\begin{equation}  \label{fBUp1:8}
\eta =  \rho  \frac {(\alpha +2) \cos (2\theta )+ 2 (1-\cos \theta )} { 2 \cos \theta } \ge \rho >0 ,
\end{equation} 
where the first inequality follows from~\eqref{eHyp2}. 
Moreover, let 
\begin{gather*} 
e(t)= E(u(t)), \\
j(t) = I(u(t)), \\
f(t)=  \|u(t)\| _{ L^2 }^2,
\end{gather*} 
where $E$ and $I$ are defined by~\eqref{fEn1} and~\eqref{fGGLc}, respectively, and  
\begin{equation} \label{fBUp1:2} 
\Est (t)= e(t) + \eta  f(t),
\end{equation} 
for $0\le t<\Tma $. 
We first claim that
\begin{equation} \label{fBUp1:12}
e(t) \le \Est (t)  < 0,
\end{equation}
for all $0\le  t <\Tma$.
Indeed, the first inequality in~\eqref{fBUp1:12} follows from~\eqref{fBUp1:2}. Moreover,   
 since $\cos (2\theta )> 0$ by~\eqref{eHyp1}, it follows from~\eqref{fGLP8}  that 
\begin{equation*} 
 \rho  \cos (2 \theta ) j(t) \le  \rho  (\alpha +2) \cos (2 \theta ) e(t) ;
\end{equation*} 
and so, we deduce from~\eqref{fGGLqbu} that
\begin{equation*}
\frac {de} {dt}   \le  
 - \cos \theta \ \| u_t   \| _{ L^2 }^2 + \rho ^2 \cos \theta \ f  + \rho  (\alpha +2) \cos (2 \theta ) e  , 
\end{equation*} 
i.e., using~\eqref{fBUp1:2},
\begin{equation} \label{fBUp1:3}
\frac {de} {dt} \le  
 - \cos \theta \ \| u_t (t) \| _{ L^2 }^2 + \rho  (\alpha +2) \cos (2 \theta ) \Est  
 + [\rho ^2 \cos \theta - \eta \rho (\alpha +2) \cos (2\theta )]  f .
\end{equation} 
Note also that by~\eqref{fGGLs} and~\eqref{fGGLsbu} 
\begin{equation} \label{fBUp1:4}
\frac {df} {dt} \le  -2 \rho f + 2 \cos \theta \ | j |  
\le  -2 \rho (1- \cos \theta ) f  + 2 \cos \theta  \Bigl| \int  _{ \R^N  } \overline{u} u_t  \Bigr| .
\end{equation} 
Since 
\begin{equation*} 
2 \Bigl| \int  _{ \R^N  } \overline{u} u_t  \Bigr| \le 2 \|u_t \| _{ L^2  } f  ^{\frac {1} {2}}
\le \frac {1} {\eta}  \| u_t \| _{ L^2 }^2 + \eta f ,
\end{equation*} 
we deduce from~\eqref{fBUp1:3} and~\eqref{fBUp1:4}  that
\begin{equation}  \label{fBUp1:5}
\frac {d \Est} {dt}   \le \rho  (\alpha +2) \cos (2 \theta ) \Est    + A   f ,
\end{equation} 
where $A =  ( \rho ^2 +\eta ^2) \cos \theta - \eta \rho [ (\alpha +2) \cos (2\theta )-  2  (1- \cos \theta ) ]   $. Note that by~\eqref{fBUp1:8} 
\begin{equation*} 
A = ( \rho ^2 +\eta ^2) \cos \theta  - 2 \eta^2 \cos \theta = ( \rho ^2 -\eta ^2) \cos \theta \le 0;
\end{equation*} 
and so  we deduce from~\eqref{fBUp1:5}  that
\begin{equation*}
\frac {d \Est} {dt}   \le   \rho  (\alpha +2) \cos (2 \theta ) \Est  .
\end{equation*} 
Therefore, $ \Est (t)  \le   e^{t \rho  (\alpha +2) \cos (2 \theta )} \Est (0)<0$, which proves the claim~\eqref{fBUp1:12}.

We now use the energy inequality~\eqref{fBUp1:12} to obtain a differential inequality on $f$. 
Observe that by~\eqref{fGGLs}, 
\begin{equation*} 
- \frac {\alpha +2} {2} e   \Bigl[  \frac {df} {dt}+ 2 \rho  f  \Bigr]
 =(\alpha +2)  \cos \theta \ e  j  .
\end{equation*} 
Since $0 >(\alpha +2) e  \ge j $ by~\eqref{fBUp1:12}  and~\eqref{fGLP8}, we deduce that  
\begin{equation} \label{fTR1} 
- \frac {\alpha +2} {2} e   \Bigl[  \frac {df} {dt}+ 2 \rho  f  \Bigr] \le  \cos \theta  \ j ^2  .
\end{equation} 
Note that by~\eqref{fGGLsbu1},
\begin{equation*} 
j ^2  =  \Bigl| \int  _{ \R^N  } \overline{u}u_t  \Bigr|^2 +\rho  ^2 f^2 + \rho  f \frac {df} {dt}  \le f  \| u_t \| _{ L^2 }^2 +\rho ^2 f^2 +\rho  f \frac {df} {dt}.
\end{equation*}  
Therefore, it follows from~\eqref{fTR1} that
\begin{equation}  \label{fTR2} 
\cos \theta  f  \| u_t \| _{ L^2 }^2 \ge - \frac {\alpha +2} {2} e(t) \Bigl[  \frac {df} {dt}+ 2 \rho  f  \Bigr] -\rho  ^2 \cos \theta  f^2 - \rho  \cos \theta  f \frac {df} {dt} .
\end{equation} 
On the other hand, multiplying~\eqref{fGGLq} by $f$ we obtain
\begin{equation*} 
- f  \frac {de} {dt}  - \frac {\rho  } {2} \cos \theta \ f \frac {df} {dt} -\rho  \sin ^2 \theta \ fj =   \cos \theta \ \| u_t \| _{ L^2 }^2 f.
\end{equation*} 
Applying~\eqref{fTR2}, we deduce that 
\begin{multline*} 
- f  \frac {de} {dt}  - \frac {\rho  } {2} \cos \theta \ f \frac {df} {dt} -\rho  \sin ^2 \theta \ fj  \\
\ge - \frac {\alpha +2} {2} e \Bigl[  \frac {df} {dt}+ 2 \rho  f  \Bigr] -\rho  ^2 \cos \theta \ f^2  -\rho  \cos \theta  \ f \frac {df} {dt},
\end{multline*} 
i.e.,
\begin{equation}  \label{fTR3} 
- f \frac {de} {dt} + \frac {\alpha +2} {2} e \frac {df} {dt} + \frac {\rho  } {2} \cos \theta \ f \frac {df} {dt} 
\ge \rho  \sin ^2 \theta \ f j -\rho  (\alpha +2) e f  -\rho  ^2 \cos \theta \ f^2  .
\end{equation} 
Since
\begin{equation*}
 j   = - \frac {\rho  } {\cos \theta } f  - \frac {1} {2\cos \theta } \frac {df} {dt} ,
\end{equation*} 
by~\eqref{fGGLs},
it follows from~\eqref{fTR3} that
\begin{multline*} 
- f \frac {de} {dt} + \frac {\alpha +2} {2} e \frac {df} {dt} + \frac {\rho  } {2} \cos \theta \ f \frac {df} {dt}  \\
\ge -  \frac {\rho  ^2  \sin ^2 \theta} {\cos \theta } f^2  - \frac { \rho  \sin ^2 \theta} {2\cos \theta } f  \frac {df} {dt}  
-\rho  (\alpha +2) e f -\rho  ^2 \cos \theta  f^2    \\ 
= - \frac {\rho  ^2 } {\cos \theta }f^2  - \frac { \rho  \sin ^2 \theta} {2\cos \theta } f  \frac {df} {dt} -\rho  (\alpha +2) e  f ;
\end{multline*} 
and so,
\begin{equation}  \label{fTR4} 
- f \frac {de} {dt}  + \frac {\alpha +2} {2} e \frac {df} {dt} +  \frac { \rho   } { 2 \cos \theta} f \frac {df} {dt}  
\ge - \frac {\rho  ^2 } {\cos \theta }f^2  -\rho  (\alpha +2) e f .
\end{equation} 
Since $e = \Est - \eta f \le - \eta f$
by~\eqref{fBUp1:2} and~\eqref{fBUp1:12}, and $\eta \ge \rho $ by~\eqref{fBUp1:8}, we see that
\begin{equation*} 
- \rho  (\alpha +2) e f    \ge  \eta \rho   (\alpha +2)    f^2 \ge \rho ^2 (\alpha +2)    f^2 .
\end{equation*} 
Therefore, we deduce from~\eqref{fTR4} that
\begin{equation}  \label{fTR4b1} 
- f \frac {de} {dt}  + \frac {\alpha +2} {2} e \frac {df} {dt} +  \frac { \rho  } { 2 \cos \theta} f \frac {df} {dt}  
\ge  \Bigl[ \alpha +2 - \frac {1} {\cos \theta }  \Bigr] \rho ^2 f^2.
\end{equation} 
Furthermore, it follows from~\eqref{fMZ2} that
\begin{equation*} 
(\alpha +2) \cos \theta \ge 2 ,
\end{equation*}
and we deduce from~\eqref{fTR4b1}  that 
\begin{equation}  \label{fTR4b2} 
- f \frac {de} {dt}  + \frac {\alpha +2} {2} e \frac {df} {dt} +  \frac { \rho   } { 2 \cos \theta} f \frac {df} {dt}  
\ge \frac {1} {\cos \theta }\rho ^2 f^2 \ge  0  .
\end{equation} 
Multiplying~\eqref{fTR4b2} by $f^{ - { \frac {\alpha +4} {2} }}>0$ we obtain
\begin{equation} \label{fTR5b1} 
\frac {d} {dt} \Bigl[ - f^{-\frac {\alpha +2} {2}} e -   \frac {\rho   } {\alpha \cos \theta} f^{-\frac {\alpha } {2}} \Bigr] \ge 0. 
\end{equation} 
Set
\begin{equation} \label{fTR6} 
\zeta = - f (0) ^{-\frac {\alpha +2} {2}} e (0) -   \frac {\rho  } {\alpha \cos \theta} f (0) ^{-\frac {\alpha } {2}},
\end{equation} 
and note that by~\eqref{eHyp3}-\eqref{eHyp4},
\begin{equation} \label{fTR11}
\zeta >0.
\end{equation} 
Integrating~\eqref{fTR5b1} on $(0,t)$, we obtain 
\begin{equation} \label{fTR7} 
- f^{-\frac {\alpha +2} {2}} e -   \frac {\rho   } {\alpha \cos \theta} f^{-\frac {\alpha } {2}}  \ge \zeta .
\end{equation} 
Multiplying~\eqref{fTR7} by $f^{ \frac {\alpha +2} {2}}$ yields
\begin{equation} \label{fTR8} 
-  e     \ge \zeta f^{ \frac {\alpha +2} {2}} +  \frac {\rho  } {\alpha \cos \theta} f   . 
\end{equation} 
On the other hand, it follows from~\eqref{fGGLs} and~\eqref{fGLP8} that
\begin{equation} \label{fTR9}
\frac {df} {dt}  \ge  -2\rho  f  - 2(\alpha +2)\cos \theta \  e . 
\end{equation} 
We deduce from~\eqref{fTR9} and~\eqref{fTR8} that
\begin{equation} \label{fTR10}
\frac {df} {dt} \ge   \frac {4\rho  } {\alpha } f + 2(\alpha +2)\cos \theta \  \zeta f^{ \frac {\alpha +2} {2}}
\ge 2(\alpha +2)\cos \theta \ \zeta f^{ \frac {\alpha +2} {2}}. 
\end{equation} 
It follows easily from~\eqref{fTR11} that  $f$ cannot satisfy~\eqref{fTR10} for all $t>0$, 
so that $\Tma <\infty $. 
This completes the proof.

\end{document}